\documentclass[12pt]{amsart}
\usepackage{amscd,amsmath,amsthm,amssymb}
\usepackage{color}
\usepackage{stmaryrd}
\usepackage{tikz}

 \def\NZQ{\mathbb}               
 
 \def\QQ{{\NZQ Q}}
 \def\ZZ{{\NZQ Z}}

 \def\G{{\mathcal G}}

 \def\Bc{{\mathcal B}}
 \def\P{{\mathcal P}}
  \def\Pc{{\mathcal P}}

 \def\ab{{\mathbf a}}
 \def\bb{{\mathbf b}}
 \def\xb{{\mathbf x}}

 \def\cb{{\mathbf c}}
 \def\db{{\mathbf d}}
 \def\eb{{\mathbf e}}
 \def\vb{{\mathbf v}}
 \def\wb{{\mathbf w}}
 \def\eb{{\mathbf e}}
 \def\opn#1#2{\def#1{\operatorname{#2}}} 
 \opn\chara{char} \opn\length{\ell} \opn\pd{pd} \opn\rk{rk}
 \opn\projdim{proj\,dim} \opn\injdim{inj\,dim} \opn\rank{rank}
 \opn\depth{depth} \opn\grade{grade} \opn\height{height}
 \opn\embdim{emb\,dim} \opn\codim{codim}
 
 \opn\Tr{Tr} \opn\bigrank{big\,rank}
 \opn\superheight{superheight}\opn\lcm{lcm}
 \opn\trdeg{tr\,deg}
 \opn\reg{reg} \opn\lreg{lreg} \opn\ini{in} \opn\lpd{lpd}
 \opn\size{size} \opn\sdepth{sdepth}
 \opn\link{link}\opn\fdepth{fdepth}\opn\lex{lex}
 \opn\tr{tr}
 \opn\type{type}
 \opn\gap{gap}
 \opn\diam{diam}
 \opn\Mod{Mod}
 \opn\Jac{Jac}
 \opn\bigheight{bigheight}
 \opn\div{div} \opn\Div{Div} \opn\cl{cl} \opn\Cl{Cl}
 \opn\Spec{Spec} \opn\Supp{Supp} \opn\supp{supp} \opn\Sing{Sing}
 \opn\Ass{Ass} \opn\Min{Min}\opn\Mon{Mon}
 \opn\Ann{Ann} \opn\Rad{Rad} \opn\Soc{Soc}
 \opn\Im{Im} \opn\Ker{Ker} \opn\Coker{Coker} \opn\Am{Am}
 \opn\Hom{Hom} \opn\Tor{Tor} \opn\Ext{Ext} \opn\End{End}
 \opn\Aut{Aut} \opn\id{id}
 
 \opn\nat{nat}
 \opn\pff{pf}
 \opn\Pf{Pf} \opn\GL{GL} \opn\SL{SL} \opn\mod{mod} \opn\ord{ord}
 \opn\Gin{Gin} \opn\Hilb{Hilb}\opn\sort{sort}
 \opn\PF{PF}\opn\Ap{Ap}
 \opn\dist{dist}
 \opn\aff{aff}
 \opn\relint{relint} \opn\st{st}
 \opn\lk{lk} \opn\cn{cn} \opn\core{core} \opn\vol{vol}  \opn\inp{inp} \opn\nilpot{nilpot}
 \opn\link{link} \opn\star{star}\opn\lex{lex}\opn\set{set}
 \opn\width{wd}
 \opn\Fr{F}
 \opn\QF{QF}
 \opn\G{G}
 \opn\type{type}\opn\res{res}
 \opn\conv{conv}
 \opn\gr{gr}
 
 \def\pot#1#2{#1[\kern-0.28ex[#2]\kern-0.28ex]}

 \opn\dirlim{\underrightarrow{\lim}}
 \opn\inivlim{\underleftarrow{\lim}}
 \let\sect=\cap

 \let\iso=\cong
 \let\Union=\bigcup

 \def\ab{{\mathbf a}}
 \def\bb{{\mathbf b}}
 \def\xb{{\mathbf x}}

 \def\cb{{\mathbf c}}
 \def\db{{\mathbf d}}
 \def\eb{{\mathbf e}}
 
 \def\Implies{\ifmmode\Longrightarrow \else
         \unskip${}\Longrightarrow{}$\ignorespaces\fi}
 \def\implies{\ifmmode\Rightarrow \else
         \unskip${}\Rightarrow{}$\ignorespaces\fi}
 \def\iff{\ifmmode\Longleftrightarrow \else
         \unskip${}\Longleftrightarrow{}$\ignorespaces\fi}

 \let\:=\colon
 \newtheorem{Theorem}{Theorem}[section]
 \newtheorem{Lemma}[Theorem]{Lemma}

 \let\epsilon\varepsilon
 \let\kappa=\varkappa
 \opn\dis{dis}
 \def\pnt{{\raise0.5mm\hbox{\large\bf.}}}
 
 \opn\Lex{Lex}

\begin{document}
\title {Binomial  ideals  attached to finite collections of cells}

\author {J\"urgen Herzog}
\address{J\"urgen Herzog, Fachbereich Mathematik, Universit\"at Duisburg-Essen, Campus Essen, 45117
Essen, Germany} \email{juergen.herzog@uni-essen.de}

\author {Takayuki Hibi}
\address{Takayuki Hibi, Department of Pure and Applied Mathematics,
Graduate School of Information Science and Technology, Osaka
University, Suita, Osaka 565-0871, Japan}
\email{hibi@math.sci.osaka-u.ac.jp}

\author {Somayeh Moradi}
\address{Somayeh Moradi, Department of Mathematics, School of Science, Ilam University,
P.O.Box 69315-516, Ilam, Iran}
\email{so.moradi@ilam.ac.ir}

\thanks{The second author was partially supported by JSPS KAKENHI 19H00637. The third author is supported by the Alexander von Humboldt Foundation}

\begin{abstract}
We consider the ideal of inner $2$-minors $I_\Pc$ of a finite set of cells $\Pc$, which we call the cell ideal of $\Pc$. A nice interpretation for the height of an unmixed ideal $I_\Pc$, in terms of the number of cells of $\Pc$ is given.
Moreover, the coordinate rings of cell ideals with isolated singularities are determined.
\end{abstract}




\maketitle

\setcounter{tocdepth}{1}

\section*{Introduction}

Combinatorial descriptions of  height of polyomino ideals have been studied in several works. 
Qureshi~\cite{Q} proved that  for a convex polyominoe $\Pc$ the height of the polyomino ideal $I_\Pc$ is the number of cells of $\Pc$. Herzog and Madani [14] extended this result to simple
polyominoes, which by definition are the polyominoes with no holes, see \cite{HM} and \cite{ASS}. Such polyomino ideals are in particular prime. However, not all polyomino ideals are prime ideals and it is still an open  problem to identify the polyominoes $\Pc$ for which $I_\Pc$  is a prime ideal. 
In \cite{DN} the same description for height in terms of the number of cells of $\Pc$ was proved for
closed path polyominoes. 
In this paper we consider more generally cell ideals, i.e., ideals of inner $2$-minors which are attached to finite collections of cells. When any two cells of $\Pc$
are connected in $\Pc$, this ideal is just the polyomino ideal. 
In Theorem~\ref{heightprime} it is shown that $\height I_\Pc\leq c\leq \bigheight I_\Pc$,  where $c$ is the number of cells of $\Pc$. In particular, if $I_\Pc$ is an unmixed ideal, then $\height I_\Pc= c$. To this aim we use Lemma~\ref{linearalgebra} 
which determines the height of an unmixed binomial ideal $I\subset S$ in terms of the
dimension of the $\QQ$-vector space spanned by the set of integer vectors $\{\vb-\wb\in \QQ^n\:\;
\xb^\vb-\xb^\wb\in I\}$.

In the next section of this paper it is shown that when $K$ is a perfect field, and $\Pc$ is a finite set of  cells such that $I_\Pc\subset S$  is a prime ideal, then the ring $S/I_\Pc$ has an isolated singularity if and only if  $\Pc$ is an inner interval.

\section{On the height of cell ideals}


Consider $(\ZZ^2,\leq )$ as a partially ordered set with $(i,j)\leq (i',j')$ if  $i\leq i'$ and $j\leq j'$. Let $\ab,\bb\in \ZZ^2$. Then the set $[\ab,\bb]=\{\cb\in \ZZ^2: \ab\leq \cb\leq \bb\}$ is called an \textit{interval}. The interval with $\ab =(i,j)$  and  $\bb = (i',j')$ is called \textit{proper}, if $i<i'$ and $j<j'$.
A \textit{cell} is an interval of the form $[\ab,\bb]$, where $\bb=\ab+(1,1)$. The cell $C=[\ab,\ab+(1,1)]$ consists of the elements  $\ab, \ab+(0,1),\ab+(1,0)$ and $\ab+(1,1)$, which  are called the \emph{vertices} of $C$.
 We denote the set of vertices of $C$ by $V(C)$. The intervals $[\ab,\ab+(1,0)]$, $[\ab+(1,0),\ab+(1,1)]$, $[\ab+(0,1),\ab+(1,1)]$ and $[\ab,\ab+(0,1)]$ are called the \emph{edges}  of $C$, denoted $E(C)$. Each edge consists of two elements, called the \emph{corners of the edge}.  Let $[\ab, \bb]$ be a proper interval in $\ZZ^2$.  A cell $C = [\ab', \bb']$ of $\ZZ^2$ is called a cell of $[\ab, \bb]$ if $\ab \leq \ab'$ and $\bb' \leq \bb$.
Let $\Pc$ be a finite collection of cells of $\ZZ^2$.  The vertex set of $\Pc$ is $V(\Pc) = \Union_{C \in \Pc} V(C)$ and the edge set of $\Pc$ is $E(\Pc) = \Union_{C \in \Pc} E(C)$. Let $C$ and $D$ be two cells of $\Pc$.  Then $C$ and $D$ are {\em connected} in $\Pc$,  if there is a sequence of cells of $\Pc$ of the form $C = C_1, \ldots, C_m = D$ for which $C_i \sect C_{i+1}$ is an edge of $C_i$ for $i = 1, \ldots, m - 1$.

A {\em polyomino} is a finite collection $\Pc$ of cells of $\ZZ^2$ for which any two cells of $\Pc$ are connected in $\Pc$.
Let $\Pc$ be a polyomino,  and let $S = K[\{x_\ab\}_{\ab \in V(\Pc)}]$ be the polynomial ring in $|V(\Pc)|$ variables over a field $K$.  A proper interval $[\ab, \bb]$ of $\ZZ^2$ is called an {\em inner interval} of $\Pc$,  if each cell of $[\ab, \bb]$ belongs to $\Pc$.  Now, for each inner interval $[\ab, \bb]$ of $\Pc$, one introduces the binomial $f_{\ab,\bb} = x_{\ab}x_{\bb}-x_{\cb}x_{\db}$, where $\cb$ and $\db$ are the anti-diagonal corners of $[\ab, \bb]$.   The binomial $f_{\ab,\bb}$ is called an {\em inner $2$-minor} of $\Pc$.  The {\em polyomino ideal} of $\Pc$ is the binomial ideal $I_\Pc$ which is generated by the inner $2$-minors of $\Pc$.  Furthermore, we write $K[\Pc]$ for the quotient ring $S/I_\Pc$.

\medskip
The main result of this section is

\begin{Theorem}
	\label{heightprime}
	Let $\Pc$ be a finite set of cells, and let $c$ be the number of cells of $\Pc$. Then $\height I_\Pc\leq c\leq \bigheight I_\Pc$. In particular, if $I_\Pc$ is an unmixed ideal, then $\height I_\Pc= c$.
\end{Theorem}

For the proof of this theorem we use
the following lemma. 
We refer to $\bigheight I$ as to the maximal height of an associated prime ideal of $I$. An ideal $I$ is called {\em unmixed} if all $P\in \Ass(I)$ have the same height.

\begin{Lemma}
	\label{linearalgebra}
	Let $I\subset S$ be a binomial ideal,  and let $V_I$ be the $\QQ$-vector space spanned by the set of integer vectors $\{\vb-\wb\in \QQ^n\:\;
	\xb^\vb-\xb^\wb\in I\}$. Then
	\[
	\height I\leq \dim_Q V_I \leq \bigheight I.
	\]
	In particular, $\height I= \dim_Q V_I$,  if $I$ is unmixed.
\end{Lemma}

\begin{proof}
	Let $\xb=x_1\cdots x_n$. Then $S_\xb=K[x_1^\pm,\ldots, x_n^\pm]$ is the Laurent polynomial ring, and  we have $\height I\leq \height IS_\xb$. Hence, for the first inequality it suffices to show that $\height IS_\xb\leq \dim_\QQ V_I$.
	
	Note that
	\[
	IS_\xb =(1-\xb^{\vb}\:\; \vb\in V_I).
	\]
	We observe that
	\[
	(1-\xb^{\vb_2})-(1-\xb^{\vb_1})=(\xb^{\vb_2}-\xb^{\vb_1})=
	\xb^{\vb_2}(1-\xb^{\vb_1-\vb_2}).
	\]
	This shows that with $1-\xb^{\vb_1}$ and $1-\xb^{\vb_2}$, also
	$(1-\xb^{\vb_1-\vb_2})\in S_\xb$,  since $\xb^{\vb_2}$ is a unit in $S_{\xb}$. Similarly, one sees that $(1-\xb^{\vb_1+\vb_2})\in S_\xb$. Hence the integer vectors $\vb$,  which span $V_I$,  form an abelian subgroup $G$ of $\ZZ^n$.  Any abelian subgroup of $\ZZ^n$ is free. Let $\vb_1,\ldots,\vb_r$
	be a basis of $G$. Then this basis is also a  $\QQ$-basis of $V_I$, and
	\[
	IS_\xb =(1-\xb^{\vb_1},\ldots, 1-\xb^{\vb_r}).
	\]
	Now, we apply Krull's generalized principle ideal theorem, to deduce
	that $\height IS_\xb \leq r=\dim_\QQ V_I$, as desired.

	For the second inequality we notice that $\height IS_\xb\leq \bigheight I$. Thus it suffices to show that $\height (1-\xb^{\vb_1},\ldots, 1-\xb^{\vb_r})=\dim_\QQ V$.
	
	Observe that $S_\xb$ can be identified with the group ring
	$K[\ZZ^n]$, whose $K$-basis  consists of all monomials $\xb^{\ab}$ with $\ab\in \ZZ^n$. By the elementary divisor theorem there exists a basis $\eb_1,\ldots, \eb_n$ of $\ZZ^n$ and positive integers $a_1,\ldots,a_r$ such that $\vb_i=a_i\eb_i$ for $i=1,\ldots,r$. In these coordinates
	\[
	IS_\xb= (1-x_1^{a_1},\ldots,  1-x_r^{a_r})S_\xb.
	\]
	Now, consider the ideal $J=(1-x_1^{a_1},\ldots, 1- x_r^{a_r})S_r$, where  $S_r=K[x_1,\ldots,x_r]$. Let $R=S_r/J$. Since $\dim R=0$, it
	follows that $R[x_{r+1},\ldots,x_n]$ is Cohen-Macaulay  of dimension $n-r$, and  since $R[x_{r+1},\ldots,x_n]\iso S/JS$, this implies that $JS$
	is an unmixed ideal of height $r$.  Because $JS$ is unmixed, we then have
	\[
	r=\height JS=\height JS_\xb= \height IS_\xb,
	\]
	as desired.
\end{proof}

{\em Proof of Theorem~\ref{heightprime}}. Note that $V_{I_\Pc}$ is a subspace of the $\QQ$-vector space $W:=\QQ^{V(\Pc)}$. We denote by $\vb_\ab \in W$ the vector, whose $\ab$'s component is $1$, while its other components are $0$. The set of vectors $\{\vb_{\ab}\:\; \ab\in V(\Pc)\}$ is the canonical basis of $W$.

For each inner interval $[\ab,\bb]$  of $\Pc$ with  anti-diagonals $\cb$ and $\db$ we define the vector
\[
\vb_{[\ab,\bb]}= \vb_\ab+\vb_\bb-\vb_\cb-\vb_\db.
\]
It follows from the definition of $V_{I_\Pc}$ that the vectors $\vb_{[\ab,\bb]}$ span  $V_{I_\Pc}$.

If $C=[\ab,\bb]$ is a cell  of $\Pc$, then  we write $\vb_C$ for the vector $\vb_{[\ab,\bb]}$ and claim that the vectors $\vb_C$ form a $\QQ$-basis of $V_{I_\Pc}$. Together with
Theorem~\ref{linearalgebra} this claim implies  the desired conclusion.

If $[\ab, \bb]$ is an arbitrary inner interval of $\Pc$, then it is readily seen that
\[
\vb_{[\ab,\bb]}= \sum_C \vb_C,
\]
where the sum is taken over all cells in $[\ab, \bb]$.  This shows that the vectors $\vb_C$ generate $V_{I_\Pc}$.

It remains to be shown that the set of vectors $\vb_C$ with $C$ a cell of $\Pc$  is linearly independent. For this purpose  we choose any total  order on $\ZZ^2$,   extending  the partial order
$\leq $ on $\ZZ^2$ which is defined by componentwise comparison. We set $\vb_{\ab}\leq \vb_{\bb}$ when $\ab\leq \bb$. Then for any cell $C=[\ab, \bb]$,  the leading  vector in the expression of $\vb_C$  is $\vb_\bb$. Since the leading vectors of all the vectors $\vb_C$  are pairwise distinct, it follows that the vectors $\vb_C$  are linearly independent.
\qed

\section{The coordinate ring of cell ideals with isolated singularity}

Let $I=(f_1, \ldots,f_m)$ be an ideal in $S$, and let \[
A=(\partial f_i/\partial x_j)_{i=1,\ldots,m \atop j=1,\ldots,n}
\]
be the corresponding Jacobian matrix. Let $h$ be the height of $I$. 
The \textit{Jacobian ideal} of the ring $R=S/I$ is the ideal $J\subset R$ generated by the $h\times h$-minors of $A$. When $K$ is a perfect field, the ideal $J$ defines the singular locus of $R$. In other words, $R_P$ is not regular for $P\in \Spec(R)$ if and only if $J\subseteq P$, see
~\cite[Corollary~16.20]{Ei}.

In the following result we investigate when the ring $K[\Pc]$ has an isolated singularity.

\begin{Theorem}
\label{isolated}
Let $K$ be a perfect field, and let $\Pc$ be a finite set of  cells such that $I_\Pc\subset S$  is a prime ideal. Then $S/I_\Pc$ has an isolated singularity if and only if  $\Pc$ is an inner interval.
\end{Theorem}

\begin{proof} We set $u_\ab=x_\ab\mod I_\Pc$ for all $\ab\in V(\Pc)$.  Let $J\subset R$ be the Jacobian ideal of  $R=S/I_\Pc$. By ~\cite[Corollary~16.20]{Ei} the assumption on $K$ guarantees that the  $K$-algebra $R$ has an isolated singularity if and only $\dim R/J=0$. The latter  is the case if and only if  suitable powers of the $K$-algebra generators  $u_\ab$ of $R$ belong to $J$.

Let us first assume that $\P$ is the inner interval $[\ab,\bb]$. The desired result follows in this case from known results about determinantal ideals, and the fact that ideal of $2$-minors of an inner interval is a determinantal ideal. 

Conversely, assume  that $R=S/I_\Pc$ has an isolated singularity. Let $h$ denotes
the number of cells in $\Pc$. Since the entries of the Jacobian matrix are of the form $0$ or
$\pm x_{\ab}$,
it follows that the Jacobian ideal is generated by
monomials of degree $h$ formed by  the generators $u_\ab$ of $R$.
Thus,  if $R$ has an isolated singularity, it is required  that for each $\ab\in V(\Pc)$,  there exists $k\geq h$  such that  $u_{\ab}^k\in J$, and this is the case if and only if $u_{\ab}^h\in J$.

Let $\ab\in V(\Pc)$.  Then $\pm x_{\ab}$ appears as an entry of the Jacobian matrix, if and only if there exists $\bb\in V(\Pc)$ such that $\ab$ and $\bb$ are the diagonal or anti-diagonal corners of an inner interval $D$ of $\Pc$.
Let $\Bc_\ab$ be the set of such elements $\bb$. Thus, if  $u_\ab^h$ appears as a monomial  generator of the Jacobian ideal $J$, then there should exists at least $h$ such elements $\bb$ so that $\ab$ and $\bb$ are the diagonal or anti-diagonal corners of an inner interval of $\Pc$. Hence, $h\leq |\Bc_\ab|$.

For each  $\bb\in \Bc_\ab$ there exists  a unique cell $C_\bb\subseteq D$ for which $\bb$ is a corner of $D$.
\begin{figure}[h]
\begin{center}
\includegraphics[scale=0.9]{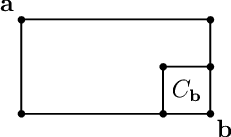}
\end{center}
\caption{Inside cell}
\label{insidecell}
\end{figure}

It follows that $|\Bc_\ab|\leq h$ (which is the number of cells of $\Pc$). Thus we have shown that $|\Bc_\ab|=h$ for all $\ab\in V(\Pc)$.  Assume that $\Pc$ is not an interval.  We claim that in this case there exists
$\ab \in \Pc$ such that $|\Bc_\ab|<h$, which then  leads to a contradiction.

Proof of the claim: choose $\ab\in V(\Pc)$, and take the subset $\{\bb_1,\ldots \bb_r\}$ of the elements in $\Bc_\ab$ for which the inner interval $I_j$ with corners $\ab$ and $\bb_j$ (as displayed in Figure~\ref{bi}) is maximal in the sense that if $\bb\in\Bc_\ab$, then the inner interval with corners $\ab$ and $\bb$ is contained in one of the intervals $I_j$.
\begin{figure}[h]
\begin{center}
\includegraphics[scale=0.9]{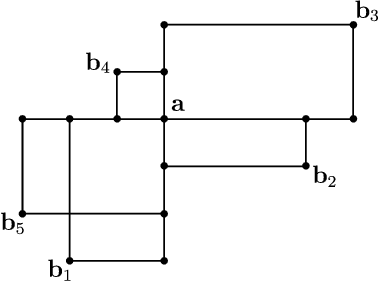}
\end{center}
\caption{}
\label{bi}
\end{figure}
Since $|\Bc_\ab|=h$ and since the cells $C_\bb$ are pairwise distinct,  and since they are  cells of $\Union_{j=1}^rI_j$, it follows that  $\Union_{j=1}^rI_j$ contains $h$ cells. By Theorem~\ref{heightprime}, $\Pc$ has exactly $h$ cells. Hence we see that $\Pc$ is equal to the set of the cells of $\Union_{j=1}^rI_j$. Let $[\cb,\db]$ be the smallest interval containing $\Pc$, see Figure~\ref{bounded}.

\begin{figure}[h]
\begin{center}
\includegraphics[scale=0.9]{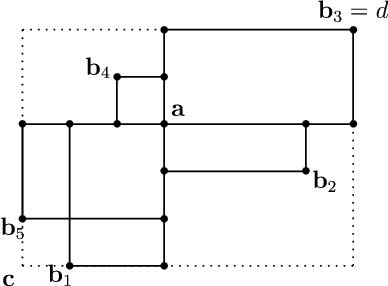}
\end{center}
\caption{}
\label{bounded}
\end{figure}

Since we assume that $\Pc$ is not an interval, it follows that not all corners of   $[\cb,\db]$ belong to $V(\Pc)$. We may assume that $\cb\not\in V(\Pc)$,  and in order to simplify our discussion we may further assume that $\cb=(0,0)$. Let $\bb$ be the smallest element on the $x$-axis and $\bb'$ be the smalest element on the $y$-axis  which belongs to $V(\Pc)$. In our picture these are the elements $\bb=\bb_1$ and $\bb'=\bb_5$.  Then  $\bb'+(1,1)\not \in \Bc_\bb$, which implies that $ |\Bc_\bb|<h$. This proves the claim and completes the proof of the theorem.
\end{proof}

{}
\end{document}